\documentclass[a4paper,draft,reqno,12pt]{amsart}
\usepackage[english]{babel}
\usepackage{amsmath}
\usepackage{amssymb}
\usepackage{amscd}
\usepackage{amsthm}
\usepackage{euscript}
\newtheorem{prop}{Proposition}
\newtheorem{lem}{Lemma}
\newtheorem{theor}{Theorem}
\newtheorem{quest}{Question}

\theoremstyle{definition}
\newtheorem{de}{Definition}
\newtheorem{ex}{Example}
\theoremstyle{remark}
\newtheorem {re}{Remark}

\DeclareMathOperator{\Aut}{Aut}

\def\Ker{{\rm Ker}\,}

\def\HH{{\mathbb H}}
\def\GA{{\mathbb G}_a}
\def\VV{{\mathbb V}}
\def\GG{{\mathbb G}}

\def\CC{{\mathbb C}}
\def\KK{{\mathbb K}}
\def\LL{{\mathbb L}}
\def\TT{{\mathbb T}}
\def\ZZ{{\mathbb Z}}

\begin{document}
\date{}
\title{On rigidity of trinomial hypersurfaces and factorial trinomial varieties}
\author{Sergey Gaifullin}
\address{Lomonosov Moscow State University, Faculty of Mechanics and Mathematics, Department of Higher Algebra, Leninskie Gory 1, Moscow, 119991 Russia; \linebreak and \linebreak
National Research University Higher School of Economics, Faculty of Computer Science, Kochnovskiy Proezd 3, Moscow, 125319 Russia}
\email{sgayf@yandex.ru}



\thanks{The author was supported by the Foundation for the Advancement of Theoretical Physics and Mathematics ``BASIS''}

\maketitle

\begin{abstract}
Trinomial varieties are affine varieties given by some special system of equations consisting of polynomials with three terms. Such varieties are total coordinate spaces of normal rational varieties with torus action of complexity one. For an affine variety $X$ we consider the subgroup $\mathrm{SAut}(X)$ of the automorphism group generated by all algebraic subgroups isomorphic to the additive group of the ground field. An affine variety $X$ is rigid if $\mathrm{SAut}(X)$ is trivial. In opposite an affine variety is flexible if $\mathrm{SAut}(X)$ acts transitively on the regular locus.  Arzhantsev proved a criterium for a factorial trinomial hypersurface to be rigid. We give two generalizations of Arzhantsev's result: a criterium for an arbitrary trinomial hypersurface to be rigid and a criterium for a factorial trinomial variety to be rigid. Also a sufficient condition for a trinomial hypersurface to be flexible is obtained. 
\end{abstract}

\section{Introduction}

Let $\KK$ be an algebraically closed field of characteristic zero and $\GA$ be its additive group. Assume $X$ is an affine algebraic variety over $\KK$. One of approaches to investigate the group of regular automorphisms of $X$ is to study $\GA$-actions on $X$. 

Let $A$ be a commutative associative algebra over $\KK$.
A linear mapping $\partial\colon A\rightarrow A$ is called a {\it derivation} if it satisfies the Leibniz rule $\partial(ab)=a\partial(b)+b\partial(a)$ for all $a,b \in A$.
A~derivation is called {\it locally nilpotent} (LND) if for any $a\in A$ there is a positive integer $n$ such that $\partial^n(a)=0$.
Exponential mapping $\partial\mapsto\{\exp(t\partial)\mid t\in\KK\}$ gives a correspondence between LNDs and algebraic
subgroups in $\mathrm{Aut}(A)$ isomorphic to $\GA$. This correspondence provides the technique for investigation of $\GA$-actions.

There are two opposite situations. If $X$ does not admit non-trivial $\GG_a$-action, the variety $X$ is called {\it rigid}. A lot of examples of rigid varieties are given in \cite{CM, FM, FMJ}, see also \cite[Chapter 10]{Fr}. 
A rigid variety has the unique maximal torus in its automorphism group~\cite{AG}. This fact sometimes allows to describe all automorphisms of $X$.  For example if $X$ is a rigid trinomial hypersurface the automorphism group of $X$ is a finite extension of the maximal torus.  
If for every regular point $x\in X$ the tangent space $T_{x}X$ is spanned by tangent vectors to 
$\GG_a$-orbits for various regular $\GG_a$-actions, then $X$ is called {\it flexible}. Flexible varieties were investigated in \cite{AFKKZ}. It was proved that flexibility is equivalent to transitivity and infinitely transitivity of the group of special automorphism on the set of regular points of $X$, see Theorem~\ref{aaa}.
Flexibility of some classes of affine varieties is proved in \cite{AFKKZ, AKZ, GSh, Pe}. In some sense flexibility of a variety means that there are a lot of $\GG_a$-actions on $X$. That is an opposite situation to rigidity of $X$.  We call {\it intermediate} such varieties that are neither rigid nor flexible.

In this paper we investigate the class of {\it trinomial varieties}. These are varieties given by systems of polynomial equations of the form 
\begin{equation}\label{eeq}
T_{01}^{l_{01}}\ldots T_{0n_0}^{l_{0n_0}}+T_{11}^{l_{11}}\ldots T_{1n_1}^{l_{1n_1}}+T_{21}^{l_{21}}\ldots T_{2n_2}^{l_{2n_2}}=0, n_0\geq 0, n_1,n_2\geq 1,
\end{equation}
see Definition \ref{trv} and \cite[Construction~1.1]{HW}. Every trinomial variety admits an algebraic torus action of complexity one, i.e. the codimension of a generic orbit equals one. The class of trinomial varieties is interesting since every normal, rational variety $X$ with only constant invertible functions, finitely generated divisor class group and an  algebraic torus action of complexity one can be obtained as a quotient of a trinomial variety via action of a diagonasible group, see \cite[Corollary~1.9]{HW}. This trinomial variety is the total coordinate space of $X$. More information on total coordinate spaces and Cox rings one can find in \cite{ADHL}.

The simplest case of a trinomial variety is a {\it trinomial hypersurface} given by a single equation of the form (\ref{eeq}) in $\KK^{n_0+n_1+n_2}$. A trinomial hypersurface is factorial if and only if one of the following holds
\begin{itemize} 
\item $n_0=0$ and $\gcd(l_{i1},\ldots,l_{in_i})=1$ for  $i=1,2$  
\item $n_0\geq 1$ and the numbers $d_i=\gcd(l_{i1},\ldots,l_{in_i})$ are pairwise coprime. 
\end{itemize}
Arzhantsev~\cite{Ar} gives a criterium for a factorial trinomial hypersurface $X$ with $n_0\geq 1$ to be rigid. He proves that $X$ is rigid if and only if all $l_{ij}\geq 2$. It is easy to see, that for arbitrary trinomial hypersurface $X$ if there is $l_{ij}=1$, then $X$ is not rigid. But the condition $l_{ij}>1$ for all  $i$ and $j$ is not sufficient for a non-factorial trinomial hypersurface to be rigid. We prove a criterium for an arbitrary trinomial hypersurface to be rigid, see Theorem~\ref{rt}.  Every hypersurface, which is not rigid and for which $l_{ij}>1$ for all $i$ and $j$, has the following form
\begin{equation}\label{fff}
\VV\left(T_{01}^2T_{02}^{2m_{02}}\ldots T_{0n_0}^{2m_{0n_0}}+T_{11}^2T_{12}^{2m_{12}}\ldots T_{1n_1}^{2m_{1n_1}}+T_{21}^{l_{21}}\ldots T_{2n_2}^{l_{2n_2}}\right).
\end{equation}
Note that we do not use Arzhantsev's result and thus reprove it as a particular case of Theorem~\ref{rt}. In \cite{GZ} and \cite{Z} all finely homogeneous (i.e. homogeneous with respect to the finest graduation such that all $T_{ij}$ are homogeneous) LNDs of  the algebra of regular functions on a trinomial hypersurface with $n_0\geq 1$ are described.  This also gives a reproving of the result~\cite{Ar} by methods similar to methods of the present paper. But these results do not imply Theorem~\ref{rt} since in non-factorial case existence of an LND does not imply existence of a finely homogeneous LND.

In Section~\ref{sec} we investigate rigidity of trinomial varieties. We give the second generalization of Arzhantsev's result to the case of factorial trinomial varieties. We prove that a factorial trinomial variety is not rigid if and only if in each monomial except at worst one for trinomial varieties of Type 1 and except at worst two for trinomial varieties of Type 2 there is a pair $(i,j)$ such that $l_{ij}=1$, see Theorem~\ref{ridfac}.

In the last two sections we investigate nonrigid trinomial hypersurfaces. A trinomial hypersurface $X$ admits an action of algebraic torus $T$ of complexity one. For nonrigid trinomial hypersurfaces we can explicitly construct some $T$-homogeneous LNDs, see Lemmas \ref{nenul} and \ref{dve}. This allows to prove a sufficient condition of trinomial hypersurface to be flexible, see Theorem~\ref{fltr}. For two classes of trinomial hypersurfaces we prove that they are intermediate, see Propositions~\ref{m} and~\ref{mm}. 

The author is grateful to Ivan Arzhantsev, Roman Budylin, Alexander Perepechko, and Yulia Zaitseva for useful discussions.

\section{Derivations}

Let $F$ be an abelian group.
An algebra $A$ is called {\it $F$-graded} if $$A=\bigoplus_{f\in F}A_f,$$ and $A_fA_g\subset A_{f+g}$.
A derivation $\partial\colon A\rightarrow A$ is called {\it $F$-homogeneous of degree $f_0\in F$} if for every $a\in A_f$, we have $\partial(a)\in A_{f+f_0}$.

Let A be a finitely generated $\ZZ$-graded algebra. Let $\partial$ be a derivation of $A$. It is easy to see, that there exists a decomposition $\partial=\sum_{i=l}^k\partial_i$ into a sum of homogeneous derivations, $\deg \partial_i=i$.  
The following lemma can be found in \cite{Re}, see also \cite{FZ}.
\begin{lem}\label{fl}
Let $\partial$ be an LND and $\partial=\sum_{i=l}^k\partial_i$, where $\partial_i$ is a homogeneous derivation of degree $i$. Then the extreme summands  $\partial_l$ and $\partial_k$ are LND.
\end{lem}

\begin{re}
Further when we wright $\partial=\sum_{i=l}^k\partial_i$, we assume that $\partial_l\neq 0$ and $\partial_k\neq 0$.
\end{re}

Lemma \ref{fl} implies the following.

\begin{lem}\label{flz}
If $A$ admits a nonzero LND, then $A$ admits a nonzero $\ZZ$-homogeneous LND. 
\end{lem}
Let $F\cong\ZZ^n$. Assume $A$ is a finitely generated $F$-graded algebra. Applying the result of Lemma~\ref{flz} $n$ times we obtain the following lemma.
\begin{lem}\label{flza}
 An $F$-graded algebra $A$ admits an LND if and only if $A$ admits an $F$-homogeneous LND. 
\end{lem}

Let $A=\KK[X]$ be the algebra of regular functions on an affine algebraic variety $X$. The group $\mathfrak{X}(T)$ of characters of $n$-dimensional torus $T=(\KK^\times)^n$ is isomorphic to $\ZZ^n$. Then $\ZZ^n$-gradings on $\KK[X]$ are in bijection to regular $T$-actions on $X$. 
If we have a $\mathfrak{X}(T)$-homogeneous derivation on $\KK[X]$, we call it a {\it $T$-homogeneous} derivation.

Let $F$ and $S$ be abelian groups. Assume $A=\bigoplus_{f\in F}A_f$ is an $F$-graded algebra. Let $\pi\colon F\rightarrow S$ be a homomorphism. 
Then $A$ can be considered as $S$-graded algebra by the following way  
$$A=\bigoplus_{s\in S}A_s, \qquad\text{where}\ A_s=\bigoplus_{\pi(f)=s}A_f.$$ 

The proof of the following lemma can be found in \cite{G}.

\begin{lem}  \label{grub}
Assume $\partial$ is an $F$-homogeneous derivation of degree $f_0$. Then $\partial$ is an $S$-homogeneous derivation of degree $s_0=\pi(f_0)$.
\end{lem}

An affine algebraic variety $X$ is called {\it rigid} if the algebra $\KK[X]$ admits no nonzero LNDs.
An affine algebraic variety $X$ is called {\it flexible} if for every regular point $x\in X^{\mathrm{reg}}$, the tangent space $T_{x}X$ is spanned by tangent vectors to 
$\GA$-orbits for various regular $\GA$-actions.
The subgroup of $\Aut(X)$ generated by all algebraic subgroups isomorphic to $\GA$ is called {\it subgroup of special automorphisms}. We denote it by $\mathrm{SAut}(X)$.

Let a group $G$ act on a set $X$. This action is called {\it $m$-transitive} if for every two $m$-tuples $(a_1,\ldots, a_m)$ and $(b_1, \ldots, b_m)$, where $a_i\neq a_j$ and $b_i\neq b_j$ if $i\neq j$, there is an element $g$ in $G$ such that for all $i$ we have $g\cdot a_i=b_i$.

If an action is $m$-transitive for every positive integer $m$, then it is called {\it infinitely transitive}.

One of the main results on flexible varieties is the following.
\begin{theor}\label{aaa}\cite[Theorem~0.1]{AFKKZ}
For an irreducible affine variety $X$ of dimension $\geq 2$, the following
conditions are equivalent.

(i) The group $\mathrm{SAut}(X)$ acts transitively on $X^{\mathrm{reg}}$,

(ii) The group $\mathrm{SAut}(X)$ acts infinitely transitively on $X^{\mathrm{reg}}$,

(iii) The variety $X$ is flexible.
\end{theor}

\begin{de}
The subring $\mathrm{ML}(X)\subset \KK[X]$, which is equal to the intersection of kernels of all LNDs of $\KK[X]$, is called the {\it Makar-Limanov invariant} of~$X$.  
\end{de}

It is easy to see, that $X$ is rigid if and only if $\mathrm{ML}(X)= \KK[X]$. From the other hand if $X$ is flexible, then $\mathrm{ML}(X)=\KK$. But the condition $\mathrm{ML}(X)=\KK$ does not imply flexibility of $X$, see \cite[$\S$ 4.2]{Li}.

\section{$m$-suspensions}\label{tre}
In this section we elaborate technique, which we use in the next sections. The main concept of this section is $m$-suspension, defined in~\cite{G}.

\begin{de}
Fix a positive integer number $m$. Let $X$ be an affine variety. Given a nonconstant regular function $f\in\KK[X]$ and positive integers $k_1,\ldots,k_m$ we define a new affine variety 
$$Y=\mathrm{Susp}(X,f,k_1,\ldots, k_m)=\VV\left(y_1^{k_1}y_2^{k_2}\ldots y_m^{k_m}-f(x)\right)\subset \KK^m\times X,$$
called an {\it $m$-suspension} over $X$ with weights $k_1,\ldots,k_m$.
\end{de}

We have a natural linear action of $m$-dimensional algebraic torus $(\KK^\times)^m$ on $m$-dimensional affine space with coordinates $y_1,\ldots,y_m$. The stabilizer $\HH$ of the monomial $y_1^{k_1}y_2^{k_2}\ldots y_m^{k_m}$ is isomorphic to the direct product of the ($m-1$)-dimensional torus $\TT$ and a finite group $\ZZ_{\gcd(k_1,\ldots,k_m)}$. We have an effective action of $\HH$ on $Y=\mathrm{Susp}(X,f,k_1,\ldots, k_m)$. 

The variety $Y$ can be reducible. For example, let $X$ be a line, then $\KK[X]=\KK[x]$. Consider the $m$-suspension  $Y=\mathrm{Susp}(X,x^2,2,\ldots, 2)$. It is given in $\KK^{m+1}$ by the unique equation $y_1^2\ldots y_m^2=x^2$. Hence $Y$ is reducible. 
During this section we assume that $Y$ is irreducible.

\begin{lem}\cite[Lemma~3.4]{G}\label{edin}
Let $\partial$ be a $\TT$-homogeneous LND on $\KK[Y]$. Then there is $i\in\{1,2,\ldots,m\}$ such that for all $j\neq i $ we have $\partial(y_j)=0$.
\end{lem}

Consider the field $\LL_j=\overline{\KK(y_j)}$, which is the algebraic closure of $\KK(y_j)$. If there is a fields embedding  $\KK\subset\LL$ and $Z$ is an affine algebraic variety over $\KK$, we denote by
$Z(\LL)$ the affine algebraic variety over $\LL$ given by the same equations as $Z$. Then we have 
$$Y(\LL_j)=\VV\left(y_1^{k_1}\ldots y_m^{k_m}-f\right)\subset \LL_j^{m-1}\times X(\LL_j).$$ 
Since $y_j$ is a constant now, $Y(\LL_j)$ is a $(m-1)$-suspension over $X(\LL_j)$.

Denote by $Y_j$ the affine algebraic variety over $\LL_j$ given by 
$$\VV\left(y_1^{k_1}\ldots y_{j-1}^{k_{j-1}}y_{j+1}^{k_{j+1}}\ldots y_m^{k_m}-f\right)\subset \LL_j^{m-1}\times X(\LL_j). $$
It is easy to see, that $Y_j\cong Y(\LL_j)$.

\begin{re}
The variety $Y_j$ can be reducible.
\end{re}

The $\TT$-action on $Y$ corresponds to a $\ZZ^{m-1}$-grading of $\KK[Y]$. Often it is convenient to consider $\ZZ$ -gradings, which are its coarsenings. Let us denote by $\mathfrak{h}_{ij}$ the $\ZZ$-grading given by $\deg(y_i)~=~-~k_j$, $\deg(y_j)~=~k_i$. Degrees of all other $y_n$ and degrees of all $g\in\KK[X]$ are equal to zero.

\begin{lem}\cite[Lemma~3.7]{G}\label{osn}
If $Y_i$ is a rigid variety, then for any nonzero $\TT$-homogeneous LND $\partial\colon\KK[Y]\rightarrow\KK[Y]$ we have 

1) $\partial(y_i)\neq 0$, 

2) $\forall j\neq i, \partial(y_j)=0$, 

3) $\deg(\partial)>0$ with respect to $\ZZ$-grading $\mathfrak{h}_{ij}$.
\end{lem}

\begin{lem}\label{ridsus}
Let $1\leq i\neq j\leq m$. Assume $Y_i$ and $Y_j$ are rigid varieties. Then $Y$ is rigid.
\end{lem}
\begin{proof}
By Lemma~\ref{osn}, we obtain that any $\TT$-homogeneous LND on $\KK[Y]$ is zero. So, by Lemma~\ref{flza} the variety $Y$ is rigid.  
\end{proof}

\begin{lem}\label{odddd}
Let $Y=\mathrm{Susp}(X,f, k_1=2, k_2=a, k_3,\ldots, k_m)$, where $a$ is an odd number. Suppose $Y_1$ is rigid. Let $\partial$ be a $\TT$-homogeneous LND of $\KK[Y]$. Then for every $g\in\KK[X]$ we have 
$y_1\mid\partial(g)$.
\end{lem}
\begin{proof}
Consider the grading $\mathfrak{h}_{12}$. Since $\partial$ is $\TT$-homogeneous, it is $\mathfrak{h}_{12}$-homogeneous. By Lemma~\ref{osn} we have $\deg \partial>0$, 
$\partial(y_1)\neq 0$, and 
$\partial(y_2)=0$. 
If $\deg \partial<a$, then $\deg(\partial(y_1))<0$. Hence $y_1\mid \partial(y_1)$. This gives a contradiction with $\partial(y_1)\neq 0$, see~\cite[Principle~5]{Fr}.
If $\deg \partial>a$, then $\partial=y_2\gamma$, where $\gamma$ is a $\mathfrak{h}_{12}$-homogeneous LND of $\KK[Y]$. 
Let us change $\partial$ to $\delta=\frac{\partial}{y_2^{k}}$ for maximal possible~$k$. Then $\deg\delta=a$. Since $\deg g=0$, we have $\deg(\delta(g))=a$. 
Since $a$ is odd, we obtain $y_1\mid \delta(g)$.
\end{proof}

\begin{lem}\label{tool}
Let 
$$Y=\mathrm{Susp}(X,f,k_1, k_2,\ldots,k_m)\ \text{and}\  Z=\mathrm{Susp}(X,f,ck_1, k_2,\ldots,k_m)$$ 
for some positive integer $c$.
Suppose $Z$ is rigid. Then for all $ j\geq 2$ we have $y_j\in\mathrm{ML}(Y)$.

\end{lem}
\begin{proof}
Let us fix $2\leq j\leq m$. Consider the $\ZZ$-grading $\mathfrak{h}_{1j}$ of $\KK[Y]$. Let $\partial$ be an $\mathfrak{h}_{1j}$-homogeneous LND of $\KK[Y]$. 

Suppose $\deg \partial\leq 0$. 
Then $\deg(\partial(y_1))<0$. Therefore, $y_1\mid \partial(y_1)$. Hence $\partial(y_1)=0$, see~\cite[Principle~5]{Fr}. Let us put $y_1=w^c$. Then $\partial$ induces a nonzero LND on 
$$\{w^{ck_1}y_2^{k_2}\ldots y_m^{k_m}=f\}\cong Z,$$ 
a contradiction. Therefore, $\deg \partial> 0$.

Now let $\delta$ be a nonzero LND of $\KK[Y]$. Consider $\delta=\sum_{i=l}^k\delta_i$ the decomposition into homogeneous summands. Then $\delta_l$ is a nonzero LND of degree $l$.  Therefore, $l> 0$. Then $y_j\mid \delta(y_j)$. Since $\delta$ is an LND, we obtain $\delta(y_j)=0$. So, $y_j\in\mathrm{ML}(Y)$. 
\end{proof}

\begin{lem}\label{prodol}
Let $X$ be a nonrigid affine variety. Then 
$$Y=\mathrm{Susp}(X,f,1,k_2,\ldots,k_m)$$
 is not rigid.
\end{lem}
\begin{proof}
Since $X$ is not rigid, there exists a nonzero LND $\partial$ of $\KK[X]$. Let us show that there exists a nonzero LND $\delta$ of 
$$\KK[X][y_1,\ldots,y_m]/(y_1y_2^{k_2}\ldots y_m^{k_m}-f)=\KK[Y].$$
Define 
\begin{equation*}
\delta(y_1)=\partial(f),
\end{equation*}
\begin{equation*}
\delta(y_j)=0, \qquad j\geq 2,
\end{equation*}
\begin{equation*}
\delta(g)=\partial(g)y_2^{k_2}\ldots y_m^{k_m}, \qquad g\in\KK[X].
\end{equation*} 
It is easy to see that $\delta$ is a nonzero LND of $\KK[Y]$.
\end{proof}

\begin{lem}\label{nod}
 Assume $Y=\mathrm{Susp}(X,f,k_1,\ldots,k_m)$ is not rigid. Let $d=\gcd(k_1, \ldots, k_m)$. Then $Z=\mathrm{Susp}(X,f,d)$ is not rigid.
\end{lem}
\begin{proof}
Since $Y$ is not rigid, by Lemma~\ref{flza} there is a nonzero $\TT$-homogeneous LND $\partial$ of $\KK[Y]$. 

Let us denote 
$$h=y_1^{\frac{k_1}{d}}\ldots y_m^{\frac{k_m}{d}}.$$

Consider $\mathfrak{X}(\TT)$-grading of $\KK[Y]$. We have 
$$\KK[Y]_0=\KK[Y]^{\TT}=\KK[X][h]\cong\KK[X][z]/(z^d-f).$$
If 
$$\deg(y_1^{a_1}\ldots y_m^{a_m})=\deg(y_1^{b_1}\ldots y_m^{b_m}),$$ 
then there is $p\in \ZZ$ such that
$$
\forall i: a_i-b_i=p\frac{k_i}{d}.
$$ 
Therefore, for every $\alpha\in\mathfrak{X}(\TT)$ there exist nonnegative integers 
$s_1(\alpha), \ldots, s_m(\alpha)$ such that 
$$\KK[Y]_{\alpha}=y_1^{s_1(\alpha)}\ldots y_m^{s_m(\alpha)}\KK[Y]_0.$$

Let $\beta=\deg\partial\in\mathfrak{X}(\TT)$.
Let us define $\delta\colon \KK[Y]_0\rightarrow \KK[Y]_0$ by the rule
$$ \delta(F)=\frac{\partial(F)}{y_1^{s_1(\beta)}\ldots y_m^{s_m(\beta)}}, \qquad F\in\KK[Y]_0. $$

We have

$$
\delta(FG)=\frac{\partial(FG)}{y_1^{s_1(\beta)}\ldots y_m^{s_m(\beta)}}=\frac{\partial(F)G+F\partial(G)}{y_1^{s_1(\beta)}\ldots y_m^{s_m(\beta)}}=\delta(F)G+F\delta(G).
$$

Hence, $\delta$ is a derivation of $\KK[Y]_0$.

Define $\nu_\partial\colon \KK[Y]\setminus\{0\}\rightarrow \ZZ_{\geq 0}$, where $\ZZ_{\geq 0}$ is the set of nonnegative integers. For nonzero $g\in\KK[Y]$, let $l=\nu_\partial(g)+1$ be the least integer number such that $\partial^l(g)=0$. 
By \cite[Proposition~1.9]{Fr}, $\nu_\partial$ is a degree function on $\KK[Y]$. Therefore,
$$
\nu_\partial(F)=\nu_\partial(\partial(F))+1>\nu_\partial(\partial(F))\geq \nu_\partial\left(\frac{\partial(F)}{y_1^{s_1(\beta)}\ldots y_m^{s_m(\beta)}}\right)=\nu_\partial(\delta(F)).
$$
This implies that $\delta$ is an LND of $\KK[Y]_0\cong\KK[Z]$. That is $Z$ is not rigid.
\end{proof}

\section{Rigid trinomial hypersurfaces}\label{rhs}

Let us consider an affine space $\KK^n$. We fix a partition $n=n_0+n_1+n_2$, where $n_0$ is a nonnegative integer and $n_1,n_2$ are positive integers. The coordinates on $\KK^n$ we denote by $T_{ij}$, $i\in\{0,1,2\}$, $1\leq j\leq n_i$.

\begin{de}
A {\it trinomial hypersurface} is a subvariety $X$ in $\KK^n$ given by the unique equation
$$T_{01}^{l_{01}}T_{02}^{l_{02}}\ldots T_{0n_0}^{l_{0n_0}}+T_{11}^{l_{11}}T_{12}^{l_{12}}\ldots T_{1n_1}^{l_{1n_1}}+T_{21}^{l_{21}}T_{22}^{l_{22}}\ldots T_{2n_2}^{l_{2n_2}}=0,$$
where $l_{ij}$ are positive integers.
If $n_0=0$, we obtain the variety 
$$\VV\left(1+T_{11}^{l_{11}}\ldots T_{1n_1}^{l_{1n_1}}+T_{21}^{l_{21}}\ldots T_{2n_2}^{l_{2n_2}}\right).$$ 
We call such a variety a {\it trinomial hypersurface with free term}. 
If $n_0\neq 0$ then $X$ is a {\it trinomial hypersurface without free term}. 
\end{de}

By \cite[Theorem~1.2(i)]{HW} every trinomial hypersurface is irreducible and normal.

We denote the monomial $T_{i1}^{l_{i1}}T_{i2}^{l_{i2}}\ldots T_{in_i}^{l_{in_i}}$ by $T_i^{l_i}$. So, trinomial hypersurface has the form $\VV\left(T_0^{l_0}+T_1^{l_1}+T_2^{l_2}\right).$

If $n_i=1$ and $l_{i1}=1$, then $X\cong \KK^{n-1}$. 
If for all $i$ we have $l_{i1}n_i>1$, then $X$ is factorial if and only if the numbers 
$$d_i = \gcd(l_{i1} , . . . , l_{in_i} ),\qquad i = 0, 1, 2$$ 
are pairwise coprime, see \cite[Theorem~1.1 (ii)]{HH}. A factorial trinomial hypersurface without free term is rigid if and only if all $l_{ij}\geq 2$, see \cite[Theorem~1]{Ar}.

The following theorem gives a criterium for an arbitrary trinomial hypersurface to be rigid. It is a generalization of  \cite[Theorem~1]{Ar}.

\begin{theor}\label{rt}
A trinomial hypersurface $X=\VV\left(T_0^{l_0}+T_1^{l_1}+T_2^{l_2}\right)$ is not rigid if and only if one of the following conditions holds

1) there exist $i\in\{0, 1, 2\}$ and $a\in \{1, 2, \ldots, n_i\}$ such that $l_{ia}=1$,

2) $n_0\neq 0$ and there exist $i\neq j\in\{0, 1, 2\}$ and $a\in \{1, 2, \ldots, n_i\}$, $b\in \{1, 2, \ldots, n_j\}$ such that $l_{ia}=l_{jb}=2$ and for all 
$u\in \{1, 2, \ldots, n_i\}$, $v\in \{1, 2, \ldots, n_j\}$ the numbers $l_{iu}$ and $l_{jv}$ are even.
\end{theor}
\begin{proof}
If condition 1) or 2) holds, we can find a nonzero LND of $\KK[X]$.
In case 1) we can assume $i=1, a=1$. Then there is the following derivation $\delta$ on $\KK[X]$, see \cite{Ar}:

$$
\delta(T_{11})=\frac{\partial T_2}{\partial T_{21}},\qquad \delta(T_{21})=-\frac{\partial T_1}{\partial T_{11}},\qquad
\delta(T_{pq})=0 \ \text{if}\  (1,1)\neq(p,q)\neq(2,1).
$$
It is easy to check that $\delta(T_0^{l_0}+T_1^{l_1}+T_2^{l_2})=0$ and $\delta$ is an LND.

In case 2) we can assume $i=0, a=1$, $j=1, b=1$. Denote for all $u$ and $v$ 
$$l_{0u}=2m_{0u},\qquad l_{1v}=2m_{1v}.$$
Then $X$ is given by the equation
$$
T_{01}^2T_{02}^{2m_{02}}\ldots T_{0n_0}^{2m_{0n_0}}+T_{11}^2T_{12}^{2m_{12}}\ldots T_{1n_1}^{2m_{1n_1}}+T_{21}^{l_{21}}T_{22}^{l_{22}}\ldots T_{2n_2}^{l_{2n_2}}=0.
$$
For $k=0,1$ let us denote $\sqrt{T_{k}^{l_k}}=T_{k1}T_{k2}^{m_{k2}}\ldots T_{kn_k}^{m_{kn_k}}$. 

The variety $X$ is isomorphic to the variety 
$$
Y=\VV\left(T_{0}^{l_0}-T_{1}^{l_1}-T_2^{l_2}\right).
$$
There is the following derivation $\delta$ of $\KK[Y]$:
\begin{equation*}
\delta(T_{01})=\frac{\partial T_2^{l_2}}{\partial T_{21}}\cdot\frac{\sqrt{T_1^{l_1}}}{T_{11}}, 
\qquad\qquad
\delta(T_{11})=\frac{\partial T_2^{l_2}}{\partial T_{21}}\cdot\frac{\sqrt{T_0^{l_0}}}{T_{01}},
\end{equation*}
\begin{equation*}
\delta(T_{21})=2\frac{\sqrt{T_0^{l_0}}}{T_{01}}\cdot\frac{\sqrt{T_1^{l_1}}}{T_{11}}\left(\sqrt{T_0^{l_0}}-\sqrt{T_1^{l_1}}\right),
\qquad\qquad
\delta(T_{ij})=0\  \text{if}\   j\neq 1.
\end{equation*}
One can check that $\delta\left(T_0^{l_0}-T_1^{l_1}-T_2^{l_2}\right)=0$.
Therefore, $\delta$ is a derivation of $\KK[Y]$.
It is easy to see that 
$$\delta\left(\sqrt{T_0^{l_0}}-\sqrt{T_1^{l_1}}\right)=0.$$ 
This implies $\delta^2(T_{21})=0$. Therefore, 
$$\delta^{l_{21}+1}(T_{01})=\delta^{l_{21}+1}(T_{11})=0.$$ 
That is $\delta$ is an LND.

Now let us assume that $X$ is a trinomial hypersurface, for which conditions 1) and 2) do not hold. 
Let us prove that $X$ is rigid by induction on n.
Base of induction consists of 4 particular cases.

{\it Case 1.} $X\cong\VV\left(x^a+y^b+1\right)$. Since the condition 1) does not hold, we have $a, b>1$. It is proven in \cite[Section~2]{FM} that $X$ is rigid. 

{\it Case 2.} $X\cong\VV\left(x^a+y^b+z^c\right)$. Since conditions 1) and 2) do not hold, we have $a, b, c>1$ and there is no a pair among them consisting of two 2. But \cite[Lemma~4]{KZ} states that in this conditions $X$ is rigid.

{\it Case 3.} $X\cong\VV\left(x^2+y^b+z^2w^c\right), b\neq  2, 2\nmid c.$ It is easy to see, that $X$ is a 2-suspension over an affine plane 
$X\cong\mathrm{Susp}(\KK^2,-(x^2+y^b),2,c)$. Case 2 provides rigidity of $\VV\left(x^2+y^b+w^c\right)$. Therefore, Lemma~\ref{odddd} implies that for every $\TT$-homogeneous LND $\delta$ of $\KK[X]$ we have $z\mid \delta(x)$ and $z\mid\delta(y)$.

By Lemma~\ref{osn}, $\delta(w)=0$. Let us consider $\LL=\overline{\KK(w)}$. Then $\delta$ induces an LND $\partial$ of $\LL[Y]=\LL[x,y,z]/(x^2+y^d+z^2)$. We have $z\mid \partial(x)$ and  $z\mid \partial(y)$. 

Let us consider the following $\ZZ$-grading on $\LL[Y]$ 
$$
\deg x=\deg z=d, \deg y =2.
$$
Consider the decomposition of $\partial$ into the homogeneous summands $\partial=\sum_{i=l}^k \partial_i$. Since $z$ is homogeneous, we obtain $z\mid \partial_l(x)$ and $z\mid \partial_l(y)$. By Lemma~\ref{fl}, the derivation $\partial_l$ is locally nilpotent.

If $d\mid l$, then $d\nmid (l+2)$. Therefore, $y\mid \partial_l(y)$. Hence, $\partial_l(y)=0$. Consider $\mathbb{F}=\overline{\LL(y)}$. Then $\partial_l$ induces an LND of 
$\mathbb{F}[x,z]/(x^2+z^2+1)$. But Case 1 implies that there are no nonzero LNDs on $\mathbb{F}[x,z]/(x^2+z^2+1)$. Therefore, $\partial_l$ is zero.

If $d\nmid l$, then $d\nmid (l+d)$. Therefore, $y\mid \partial_l(z)$. Hence either $\partial_l(y)=0$ or $\partial_l(z)=0$. In both cases by the same arguments as above $\partial_l$ is zero.

It is easy to see, that $\partial_l=0$ implies $\delta=0$. That is $X$ is rigid.

{\it Case 4.} $X\cong\VV\left(x^2+y^2v^b+z^2w^c\right), 2\nmid b>1, 2\nmid c>1.$ 
Let us consider the following $\ZZ^2$-grading of $\KK[X]$ 
$$
\deg x=(0,0), \deg y=(-b,0), \deg v=(2,0), \deg z=(0,-c), \deg w = (0,2).
$$
Suppose $X$ is not rigid. Then by Lemma~\ref{flza} there exists a $\ZZ^2$-homogeneous derivation $\partial$.
Again we can consider $X$ as 2-suspension over an affine space 
$$X\cong\mathrm{Susp}(\KK^3,-(x^2+y^2v^b),2,c).$$ 
Case 3 provides rigidity of $\VV\left(x^2+y^2v^b+w^c\right)$. Therefore $\partial(w)=0$ and Lemma~\ref{odddd} implies that $z\mid\partial(y)$ and $z\mid\partial(x)$. Analogically $\partial(v)=0$, $y\mid\partial(z)$ and $y\mid\partial(x)$. Hence either $\partial(z)=0$ or $\partial(y)=0$. Considering the field $\overline{\KK(v,w)}$ we obtain the same situation as in Case 3. So we obtain a contradiction, which implies that $X$ is rigid.

{\it Inductive step}. Let $X$ be a a trinomial hypersurface, for which conditions 1) and 2) do not hold. Suppose $X$ is not rigid. Note that $X$ is an $n_2$-suspension over $\KK^{n_0+n_1}$. Let us consider the varieties 
$$X_j=\VV\left(T_0^{l_0}+T_1^{l_1}+\frac{T_2^{l_2}}{T_{2j}^{l_{2j}}}\right),\qquad j=1,2,\ldots, n_2.$$ 
Since the condition 1) does not hold for $X$, it does not hold for all $X_j$. If for some $X_j$ the condition 2) does not hold, then by induction hypothesis $X_j$ is rigid. By Lemma~\ref{ridsus} if for some $i\neq j$ the varieties $X_i$ and $X_j$ are rigid, then $X$ is rigid.
Since $X$ is not rigid, there is $i\in\{1,2,\ldots,n_2\}$ such that for all $j\neq i$, for $X_j$ the condition 2) holds. We have analogical situation for the other monomials. It is easy to see that this is possible if and only if $X$ is isomorphic to varieties from Cases 3 or 4.
Theorem~\ref{rt} is proved.
\end{proof}

\begin{ex}
In \cite[Section~9.1]{FMJ} the following trinomial variety is considered
$$
X=\VV\left(t^dx^a+y^b+z^c\right).
$$
Theorem~\ref{rt} implies that $X$ is not rigid if and only if one of the following conditions hold

1) $1\in\{a, b, c, d\}$;

2) $X\cong\VV\left(t^dx^a+y^2+z^2\right)$;

3)  $X\cong\VV\left(t^{2m}x^2+y^2+z^c\right)$.

This coincides with the assertions of \cite[Theorem~9.1]{FMJ} and \cite[Remark~9.2]{FMJ}.
\end{ex}

In \cite{AG} the automorphism group of a rigid trinomial hypersurface without free term is described. So, using Theorem~\ref{rt} we obtain description of the automorphism group of some trinomial hypersurfaces.

\begin{ex}
Let $X=\VV\left(x^6y^3+z^6u^3+v^6w^3\right)$. Then $X$ is not factorial and it is rigid by Theorem~\ref{rt}. Therefore, we can apply \cite[Theorem~5.5]{AG} to describe automorphisms of $X$. We have 
$$
\mathrm{Aut}(X)\cong \mathrm{S}_3\rightthreetimes (\left(\mathbb{Z}/3\mathbb{Z}\right)\times\left(\mathbb{Z}/3\mathbb{Z}\right)\times T),
$$
where $S_3$ permutes the monomials, $T$ is the 4-dimensional torus acting by 
$$
(t_1, t_2, t_3, t_4)\cdot (x, y, z, u, v, w)=(t_1 t_4 x, t_1^{-2}y, t_2t_4 z, t_2^{-2} u, t_3t_4 v, t_3^{-2} w),
$$
the first copy of $\mathbb{Z}/3\mathbb{Z}$ acts by multiplying $u$ by $\sqrt[3]{1}$ and the second one acts by multiplying $w$ by $\sqrt[3]{1}$.
\end{ex}

\section{Trinomial varieties}
\label{sec}

Let us remind \cite[Construction~1.1]{HW}.
Fix integers $r,n>0$, $m\geq 0$, and $q\in\{0,1\}$. And fix a partition 
$$n=n_q+\ldots+n_r,\qquad n_i>0.$$ 
For each $i=q,\ldots, r$ fix a tuple $l_i=(l_{i1}, l_{i2},\ldots, l_{in_i})$ of positive integers and define a monomial
$$
T_i^{l_i}=T_{i1}^{l_{i1}}\ldots T_{in_i}^{l_{in_i}}\in\KK\left[T_{ij},S_k\mid q\leq i\leq r, 1\leq j\leq n_i, 1\leq k\leq m \right].
$$
We write $\KK[T_{ij} , S_k ]$ for the above polynomial ring. Now we define a ring $R(A)$ for some input data $A$.

{\it Type 1.}  $q=1$, $A=(a_1,\ldots, a_r)$, $a_j\in\KK$, if $i\neq j$, then $a_i\neq a_j$.
Set $I = \{1,...,r-1\}$ and for every $i \in I$ let us define a polynomial
$$
g_i=T_i^{l_i}-T_{i+1}^{l_{i+1}}-(a_{i+1}-a_i)\in\KK[T_{ij},S_k].
$$

{\it Type 2.} $q=0$,
$$
A=
\begin{pmatrix}
a_{10}&a_{11}&a_{12}&\ldots&a_{1r}\\
a_{20}&a_{21}&a_{22}&\ldots&a_{2r}
\end{pmatrix}
$$
is a $2\times (r+1)$-matrix with pairwise linearly independent columns. 
Set $I = \{0,...,r-2\}$ and define for every $i \in I$ a polynomial
$$
g_i=\det
\begin{pmatrix}
T_i^{l_i}&T_{i+1}^{l_{i+1}}&T_{i+2}^{l_{i+2}}\\
a_{0i}&a_{0i+1}&a_{0i+2}\\
a_{1i}&a_{1i+1}&a_{1i+2}
\end{pmatrix}
\in\KK[T_{ij},S_k].
$$

For both Types we define $R(A)=\KK[T_{ij},S_k]/(g_i\mid i\in I)$.

\begin{de}\label{trv}
The variety $X=\mathrm{Spec}(R(A))$ for some $A$ we call a {\it trinomial variety}.
\end{de}

By \cite[Theorem~1.2(i)]{HW} every trinomial variety is irreducible and normal.
We need the following fact.

\begin{prop}\cite[Theorem~1.2($\mathrm{iv}$)]{HW}\label{faktor}
Suppose $r \geq 2$ and $n_il_{ij} > 1$ for all $i, j$. Then 

(a) in case of Type 1, $R(A)$ is factorial if and only if one has $\gcd(l_{i1},\ldots, l_{in_i}) = 1$ for $i = 1,\ldots,r$;

(b) in case of Type 2, $R(A)$ is factorial if and only if the numbers $d_i := \gcd(l_{i1},\ldots , l_{in_i} )$ are pairwise coprime.
\end{prop}

Now let us give a criterium for a factorial trinomial variety to be rigid.

\begin{theor}\label{ridfac}
Let $X$ be a trinomial variety of Type 1 or a factorial trinomial variety of Type~2.
Then  $X$ is not rigid if and only if one of the following holds:

1) $m\neq 0$,

2) In case of Type 1 there is $b\in\{1,\ldots r\}$ such that for each $i\in\{1,\ldots r\}\setminus\{b\}$ there is $j(i)\in\{1,\ldots n_i\}$ such that $l_{ij(i)}=1$,

3) In case of Type 2 there are $b$ and $c$ in $\{0,\ldots r\}$ such that for each $ i\in\{0,\ldots r\}\setminus\{b,c\}$ there is $j(i)\in\{1,\ldots n_i\}$ such that $l_{ij(i)}=1$.
\end{theor}
\begin{proof}
If $m\neq 0$, we have an LND $\frac{\partial}{\partial S_1}$ of $R(A)$. So, in the sequel we assume $m=0$.

If $r=2$, $X$ is a trinomial hypersurface and the assertion of the theorem follows from Theorem~\ref{rt}.

Let us prove by induction on $r$ that if one of conditions 2) and 3) holds, then $X$ is not rigid. Base of induction $r=2$ is already proved.

Assume $r>2$.
If condition 2) holds, we can assume that $b=1$. Indeed, the system of equations of $X$ is not symmetric with respect to groups $T_i$. But we can replace equations by linear combinations of them in such a way that we obtain system of equations of Type 1 with permuting groups. Analogically, if the condition 3) holds, we can assume $b=0, c=1$. 

Consider 
$$\mathcal{A}=\KK\left[T_{ij}\mid q\leq i\leq r-1, 1\leq j\leq n_i \right],$$
$$Z=\mathrm{Spec}\left(\mathcal{A}/(g_i\mid i\in I\setminus\{r-1\})\right) \text{in case of Type 1},$$
$$Z=\mathrm{Spec}\left(\mathcal{A}/(g_i\mid i\in I\setminus\{r-2\})\right) \text{in case of Type 2}.$$
Then $X=\mathrm{Susp}(Z,f,l_{r1},\ldots,l_{rn_r})$, where $f=T_{r-1}^{l_{r-1}}+p$ for some $p\in\KK$ in case of Type 1 and $f=p_1T_{r-2}^{l_{r-2}}+p_2T_{r-1}^{l_{r-1}}$ for some $p_1, p_2\in\KK$ in case of Type 2. Lemma~\ref{prodol} implies that if $Z$ is not rigid, then $X$ is not rigid. This completes the inductive step.   

Now, let us prove that if neither 2) nor 3) holds, then $X$ is rigid. 

{\bf Case of Type 1.} Suppose, condition 2) does not hold, but $X$ is not  rigid. We have seen that $X$ is an $n_r$-suspension over $Z$. By Lemma~\ref{ridsus} we can decrease $m$ until $m=1$. Analogically, we can decrease the numbers of variables $T_{ij}$ for all $i$. So, we obtain a trinomial variety $ \widetilde{X}$ such that
\begin{itemize}
\item $\widetilde{X}$ does not satisfy the condition 2), 
\item $\widetilde{X}$  is not rigid,
\item for $\widetilde{X}$ we have $n_1=n_2=\ldots=n_r=1$. 
\end{itemize}
We use notations $\widetilde{T}_{ij}$ and $\widetilde{l}_{ij}$ for variables and their powers in equations of $\widetilde{X}$. If $\widetilde{l}_{ij}=1$, then $\widetilde{X}$ is isomorphic to a trinomial variety with less $r$. We assume that we have eliminated all variables $\widetilde{T}_{ij}$ with $\widetilde{l}_{ij}=1$. So now all $\widetilde{l}_{ij}\geq 2$ and since the condition 2) does not hold, $r\geq 2$.

We need the following lemma, which is Lemma~2 from \cite{LML}, see also  \cite[Theorem~2.2~(a)]{FMJ}.

\begin{lem}\label{sumdv}
Let $B$ be a commutative $\KK$-domain and $\partial$ be non-zero LND of $B$. Suppose
$u,v \in \Ker\partial$ and $x, y \in B$ are non-zero, and $a$,~$b$ are integers with $a, b \geq 2$. Assume $ux^a+vy^b\neq 0$.  If $\partial(ux^a+vy^b)=0$, then $\partial(x)=\partial(y)=0$.
\end{lem}

In our situation $B=\KK[\widetilde{X}]$, $u=1$, $v=-1$, $x=\widetilde{T}_{i1}$, $y=\widetilde{T}_{i+1,1}$, $a=\widetilde{l}_{i1}$, $b=\widetilde{l}_{i+1,1}$. Then $ux^a+vy^b=\widetilde{T}_{i1}^{\widetilde{l}_{i1}}-\widetilde{T}_{i+1,1}^{\widetilde{l}_{i+1,1}}\neq 0\in\KK$. 
Hence, $\partial(ux^a+vy^b)=0$. Then, by Lemma~\ref{sumdv}, $\partial(\widetilde{T}_{i1})=\partial(\widetilde{T}_{i+1,1})=0$. Since this is true for all~$i$, $\widetilde{X}$ is rigid. Therefore, $X$ is rigid.

{\bf Case of Type 2.} 
We use induction on $r$. Base of induction $r=2$ is already proved in~Theorem~\ref{rt}.

Since the condition 3) does not hold, there are $a,b,c\in\{0,1,\ldots,r\}$ such that for all possible $j$ we have $l_{aj}\geq 2, l_{bj}\geq 2, l_{cj}\geq 2$. We can assume that $a=0, b=1, c=2$. 
Again consider $Z$. Since $X$ is factorial any two of $d_i$ are coprime. Hence, $Z$ is factorial too. In the step of induction it is sufficiently to show that if $r\geq 3$ rigidity of $Z$ implies rigidity of $X$. 

There are two cases.

{\it Case A. $n_r=1$.}
If $l_{r1}=1$, then $X\cong Z$. So, in the sequel we assume $l_{r1}\geq 2$.

\begin{lem}\cite[Corollary~4.4]{FMJ}\label{roo}
Suppose $R$ is a $\KK$-domain, $\mathfrak{g}$ is a $\ZZ$-grading of $R$, and $f \in R$ is homogeneous of degree $ \phi\in \ZZ$ ($\phi\neq 0$). Assume that $n\geq 2$ is an integer such that $B = R[z]/(f +z^n)$ is a domain, and let $\mathfrak{h}$ be the $\ZZ$-grading of $B$ defined by $\mathfrak{h} = (n\mathfrak{g},\phi)$. If $\gcd(\phi,n) = 1$, then 
$\partial(\mathrm{ML}(R)) = 0$ for every $\mathfrak{h}$-homogeneous LND $\partial$ of~$B$.
\end{lem}

We would like to use this lemma in case when 
$$
R=\KK[Z], B=\KK[X], f=-(p_1T_{r-2}^{l_{r-2}}+p_2T_{r-1}^{l_{r-1}}), z=T_{r1}, n=l_{r1}.
$$ 
There are integers $u_{i1}, \ldots, u_{in_i}$
such that $u_{i1}l_{i1}+\ldots+u_{in_i}l_{in_i}=d_i$. Let us define a $\ZZ$-grading $\mathfrak{g}$ of $R$ by the rule
$$
\deg T_{ij}=u_{ij}d_1d_2\ldots d_{i-1}d_{i+1}\ldots d_{r-1}.
$$

Then 
$$\deg f= \prod_{k=1}^{r-1}d_i.$$

Since $n_r=1$, we have $l_{r1}=d_r$. By Proposition~\ref{faktor}, 
$$\gcd(d_r,\deg f)=\gcd\left(d_r,\prod_{k=1}^{r-1}d_i\right)=1.$$

We see that all conditions of Lemma~\ref{roo} are satisfied. Therefore, for every $\mathfrak{h}$-homogeneous LND $\partial$ we obtain $\partial(\mathrm{ML}(R)) = 0$. But we assume that $R=\KK[Z]$ is rigid. That is $\mathrm{ML}(R)=R$. We obtain that if $i\leq r-1$, then $\partial(T_{ij})=0$ for all $j$. This implies $\partial (T_{r1})=0$, so $\partial\equiv 0$. Lemma~\ref{flz} implies that $X$ is rigid.

{\it Case B. $n_r\geq 2$.}
We see that $X$ is an $n_r$-suspension over $Z$. Suppose $X$ is not rigid. By Lemma~\ref{nod} the variety
$$V=\VV\left(f-y^{d_r}\right)\subset Z\times\KK$$ 
is not rigid. If $d_r=1$ then $V\cong Z$. But $Z$ is rigid. Therefore, $d_r\geq 2$. Hence, $V$ is a factorial trinomial variety, such that condition 3) does not hold. Therefore, we are in Case A. Hence, $V$ is rigid. This implies that $X$ is rigid.
Theorem \ref{ridfac} is proved.
\end{proof}

\begin{ex}
The following subvariety $X$ in $\KK^4$ is rigid
$$
\begin{cases}
x^2+y^2+z^2=1;\\
y^2+z^2+w^2=1.
\end{cases}
$$
Indeed, it is easy to see that $X$ is isomorphic to the variety Y given by
$$
\begin{cases}
T_{21}T_{22}-T_{11}^2=1;\\
T_{31}^2-T_{21}T_{22}=1.\\
\end{cases}
$$
If we take $A=(1,2,3)$, $T_1^{l_1}=T_{11}^2$, $T_2^{l_2}=T_{21}T_{22}$ and $T_3^{l_3}=T_{31}^2$, we obtain $Y$ as a trinomial variety of Type 1.  Therefore, Theorem~\ref{ridfac} implies rigidity of $Y$.

\end{ex}

\section{Flexible trinomial hypersurfaces}

In Section~\ref{rhs} we have obtained a criterium for trinomial hypersurface to be rigid.
Now let us consider trinomial hypersurfaces which are not rigid. The opposite to rigidity situation is flexibility. Let us prove that some trinomial hypersurfaces are flexible. 
We introduce five special types of trinomial hypersurfaces, which we denote $H_1$-$H_5$.

\begin{tabular}{c|c}

{\tiny\ }&{\tiny\ }\\
$H_1$&
$\VV\left(T_0^{l_0}+T_1^{l_1}+T_{21}T_{22}\ldots T_{2n_2}\right)$\\
{\tiny\ }&{\tiny\ }\\
\hline
{\tiny\ }&{\tiny\ }\\
$H_2$&$\VV\left(T_{01}^2T_{02}^2\ldots T_{0n_0}^2+T_{11}^2T_{12}^2\ldots T_{1n_1}^2+T_2^{l_2}\right)$\\
{\tiny\ }&{\tiny\ }\\
\hline
{\tiny\ }&{\tiny\ }\\
$H_3$&$\VV\left(T_0^{l_0}+T_{11}T_{12}^{l_{12}}\ldots T_{1n_1}^{l_{1n_1}}+T_{21}T_{22}^{l_{22}}\ldots T_{2n_2}^{l_{2n_2}}\right)$\\
{\tiny\ }&{\tiny\ }\\
\hline
{\tiny\ }&{\tiny\ }\\
$H_4$&$
\VV\left(T_{01}T_{02}^{l_{02}}\ldots T_{0n_0}^{l_{0n_0}}+T_{11}^2T_{12}^{2m_{12}}\ldots T_{1n_1}^{2m_{1n_1}}+T_{21}^2T_{22}^{2m_{22}}\ldots T_{2n_2}^{2m_{2n_2}}\right)$\\
{\tiny\ }&{\tiny\ }\\
\hline
{\tiny\ }&{\tiny\ }\\
$H_5$&$\VV\left(T_{01}^2T_{02}^{2m_{02}}\ldots T_{0n_0}^{2m_{0n_0}}+T_{11}^2T_{12}^{2m_{12}}
\ldots T_{1n_1}^{2m_{1n_1}}+T_{21}^2T_{22}^{2m_{22}}\ldots T_{2n_2}^{2m_{2n_2}}\right)$\\
{\tiny\ }&{\tiny\ }\\
\end{tabular}

These types have nontrivial intersections. For example, the variety $\VV\left(T_{01}^2+T_{11}^2+T_{21}^2\right)$ belongs to Type $H_2$ and Type $H_5$.
\begin{re}
We use the same notations as in Section~\ref{rhs}. In particular, $n_0\geq 0$, $n_1,n_2>0$. Types of hypersurfaces we denote up to permutation of variables. For example, the variety $\VV\left(1+T_{11}T_{12}+T_{21}^3 \right)$ is of Type $H_1$, but the variety $\VV\left(1+T_{11}^2+T_{21}^2\right)$ is not of Type $H_2$.
\end{re}

\begin{theor}\label{fltr}
Trinomial hypersurfaces of Types $H_1$-$H_5$ are flexible.
\end{theor}

To prove this theorem we need two lemmas.

\begin{lem}\label{nenul}
Let $X=\VV\left(T_0^{l_0}+T_1^{l_1}+T_{21}T_{22}^{l_{22}}\ldots T_{2n_2}^{l_{2n_2}}\right)$. Let us fix $n_2-1$ numbers $c_2,\ldots,c_{n_2}\in \KK$. Consider the subvariety 
$$
X(c_2,\ldots,c_{n_2})=\VV\left(T_{22}-c_2, \ldots, T_{2n_2}-c_{n_2}\right)\subset X. 
$$
If for all $k\in\{2,\ldots,n_0\}$ we have $c_k\neq 0$, then $\mathrm{SAut}(X)$ acts on $X(c_2,\ldots,c_{n_2})$ transitively. 
\end{lem}
\begin{proof}
For every $i\in\{0,1\}$ and  $j\in\{1,2,\ldots, n_i\}$
we can define an LND $\gamma_{ij}$ of $\KK[X]$ by
$$
\gamma_{ij}(T_{21})=-\frac{\partial T_i}{\partial T_{ij}},\qquad
\gamma_{ij}(T_{ij})=\frac{\partial T_2}{\partial T_{21}},\qquad
\gamma_{ij}(T_{pq})=0\ \text{for all other pairs}\ (p,q).$$

Denote $\tau_{ij}(s)=\exp(s\gamma_{ij})$. Then
\begin{equation*}
\tau_{ij}(s)(T_{ij})=T_{ij}+sT_{22}^{l_{22}}\ldots T_{2n_2}^{l_{2n_2}},
\end{equation*}
$\tau_{ij}(s)(T_{pq})=T_{pq}$ for all pairs $(p,q)\neq(2,1)$, $(p,q)\neq(i,j)$.

Let $P, Q\in X(c_2,\ldots,c_{n_2})$.
Since 
$$T_{22}^{l_{22}}\ldots T_{2n_2}^{l_{2n_2}}(P)=T_{22}^{l_{22}}\ldots T_{2n_2}^{l_{2n_2}}(Q)\neq 0,$$ we can take 
$$
s_{ij}=\frac{T_{ij}(Q)-T_{ij}(P)}{T_{22}^{l_{22}}\ldots T_{2n_2}^{l_{2n_2}}(P)}.
$$
Put $R=\tau_{ij}(s_{ij})(P)$. Then $T_{ij}(R)=T_{ij}(Q)$.

Starting from $P$ we can change by $\tau_{ij}(s_{ij})$ all $T_{ij}$. As the result we get a point $S\in X(c_2,\ldots,c_{n_2})$ such that $T_{ij}(S)=T_{ij}(Q)$ for all 
$i=0,1$; $j=1,2,\ldots, n_i$. Then 
$$T_{21}(S)=-\frac{T_0^{l_0}(S)+T_1^{l_1}(S)}{T_{22}^{l_{22}}\ldots T_{2n_2}^{l_{2n_2}}(S)}=
-\frac{T_0^{l_0}(Q)+T_1^{l_1}(Q)}{T_{22}^{l_{22}}\ldots T_{2n_2}^{l_{2n_2}}(Q)}=T_{21}(Q).$$
Therefore, $S=Q$. That is $P$ and $Q$ are in one $\mathrm{SAut}(X)$-orbit.

\end{proof}

\begin{lem}\label{dve}
Consider 
$$X=\VV\left(T_{01}^2T_{02}^{2m_{02}}\ldots T_{0n_0}^{2m_{0n_0}}-T_{11}^2T_{12}^{2m_{12}}\ldots T_{1n_1}^{2m_{1n_1}}-T_2^{l_2}\right). $$ 
Let us fix $n_0+n_1-2$ numbers: $a_2,\ldots, a_{n_0}, b_2, \ldots, b_{n_1}$. Consider the subvariety 
$$
X(a_2,\ldots, a_{n_0}, b_2, \ldots, b_{n_1})=\VV\left(T_{02}-a_2, \ldots, T_{0n_0}-a_{n_0},T_{12}-b_2, \ldots, T_{1n_1}-b_{n_1}\right)\subset X. 
$$
If all $a_u\neq 0$ and all $b_v\neq 0$, then $\mathrm{SAut}(X)$ acts on $X(a_2,\ldots, a_{n_1}, b_2, \ldots, b_{n_2})\cap X^{reg}$ transitively. 
\end{lem}
\begin{proof}
For each $1\leq i\leq n_2$ we have two LNDs $\delta_{i+}$ and $\delta_{i-}$ of $\KK[X]$ given by

$$
\left\{
\begin{array}{lcl}
\delta_{i+}(T_{01})=\frac{\partial T_2^{l_2}}{\partial T_{2i}}\frac{\sqrt{T_1^{l_1}}}{T_{11}};\\
\delta_{i+}(T_{11})=-\frac{\partial T_2^{l_2}}{\partial T_{2i}}\frac{\sqrt{T_0^{l_0}}}{T_{01}};\\
\delta_{i+}(T_{2i})=2\frac{\sqrt{T_0^{l_0}}\sqrt{T_1^{l_1}}}{T_{01}T_{11}}\left(\sqrt{T_0^{l_0}}+\sqrt{T_1^{l_1}}\right);\\
\delta_{i+}(T_{jk})=0 \text{ for all } (j,k)\notin\{ (0,1), (1,1), (2,i)\}.
\end{array}
\right.
$$
$$
\left\{
\begin{array}{lcl}
\delta_{i-}(T_{01})=\frac{\partial T_2^{l_2}}{\partial T_{2i}}\frac{\sqrt{T_1^{l_1}}}{T_{11}};\\
\delta_{i-}(T_{11})=\frac{\partial T_2^{l_2}}{\partial T_{2i}}\frac{\sqrt{T_0^{l_0}}}{T_{01}};\\
\delta_{i-}(T_{2i})=2\frac{\sqrt{T_0^{l_0}}\sqrt{T_1^{l_1}}}{T_{01}T_{11}}\left(\sqrt{T_0^{l_0}}-\sqrt{T_1^{l_1}}\right);\\
\delta_{i-}(T_{jk})=0 \text{ for all } (j,k)\notin\{ (0,1), (1,1), (2,i)\}.
\end{array}
\right.
$$
As in the proof of Theorem~\ref{rt} one can check that $\delta_{i+}$ and  $\delta_{i-}$ are LNDs of $\KK[X]$.

Denote 
\begin{equation*}
\alpha=\sqrt{T_0^{l_0}}-\sqrt{T_1^{l_1}},
\end{equation*} 
\begin{equation*}
\beta=\sqrt{T_0^{l_0}}+\sqrt{T_1^{l_1}}.
\end{equation*}
Then 
\begin{equation*}
\delta_{i+}(\alpha)=\delta_{i-}(\beta)=2\frac{\partial T_2^{l_2}}{\partial T_{2i}}\frac{\sqrt{T_0^{l_0}}\sqrt{T_1^{l_1}}}{T_{01}T_{11}},
\qquad
\delta_{i+}(\beta)=\delta_{i-}(\alpha)=0.
\end{equation*}

Denote $\varphi_{i+}(s)=\exp(s\delta_{i+})$ and $\varphi_{i-}(s)=\exp(s\delta_{i-})$.
We have 
$$
\varphi_{i\pm}(s)(T_{2i})=T_{2i}+2s\frac{\sqrt{T_0^{l_0}}\sqrt{T_1^{l_1}}}{T_{01}T_{11}}\left(\sqrt{T_0^{l_0}}\pm\sqrt{T_1^{l_1}}\right).
$$

Let us fix a point $Q$ in $X(a_2,\ldots, a_{n_0}, b_2, \ldots, b_{n_1})\cap X^{reg}$ such, that for all $i$ and~$j$ we have $T_{ij}(Q)\neq 0$. Let $P$ be an arbitrary point in $X(a_2,\ldots, a_{n_0}, b_2, \ldots, b_{n_1})\cap X^{reg}$. If we can map $P$ to $Q$ by an element of $\mathrm{SAut}(X)$, then $X(a_2,\ldots, a_{n_0}, b_2, \ldots, b_{n_1})\cap X^{reg}$ is contained in one $\mathrm{SAut}(X)$-orbit. Let us construct a sequence of automorphisms in $\mathrm{SAut}(X)$, composition of which maps $P$ to $Q$.

First of all we would like to map $P$ by $\mathrm{SAut}(X)$ to such a point $R$, that $T_{2i}(R)=T_{2i}(Q)$ for all $i$. To do this we obtain points 
$$P_0=P, P_1,\ldots,P_{n_0}=R\in\left(X(a_2,\ldots, a_{n_0}, b_2, \ldots, b_{n_1})\cap X^{reg}\right),$$
where $T_{21}(P_i)=T_{21}(Q),\ldots, T_{2i}(P_i)=T_{2i}(Q)$ and $P_i=\beta_i(P_{i-1})$ for some $\beta_i\in\mathrm{SAut}(X)$.
We start from $i=1$ and $P_0=P$. When we obtain a point $P_i$, we increase $i$. The point $R$ coincides with $P_{n_0}$.

Since all $a_u\neq 0$ and all $b_v\neq 0$, if $T_{01}(P_{i-1})\neq 0$ or $T_{11}(P_{i-1})\neq 0$, then at least one $\delta_{i+}(T_{2i})(P_{i-1})$ or $\delta_{i-}(T_{2i})(P_{i-1})$ is not equal to zero. Therefore, we can change $T_{2i}(P_{i-1})$ by $\varphi_{i+}$ or by $\varphi_{i-}$. Hence, there is $s$ such that 
$T_{2i}(\varphi_{i\pm}(s)(P_{i-1}))=T_{2i}(Q)$. We can take $P_i=\varphi_{i\pm}(s)(P_{i-1})$. 

If $T_{01}(P_{i-1})=T_{11}(P_{i-1})=0$, then there is an index $k$ such that $T_{2k}(P_{i-1})=0$. For every $j<i$ we have $T_{2j}(P_{i-1})=T_{2j}(Q)\neq 0$. Therefore, $k\geq i$. Since $P_{i-1}$ is a regular point, we have $l_{2k}=1$ and $k$ is the unique index such that $T_{2k}(P_{i-1})=0$. Then we can consider the LND $\gamma$ given by
$$
\gamma(T_{2k})=\frac{\partial T_1^{l_1}}{\partial T_{11}}\qquad \gamma(T_{11})=-\frac{\partial T_2^{l_2}}{\partial T_{2k}}, \qquad
\gamma(T_{pq})=0\ \text{for all other pairs}\ (p,q).
$$
Denote $\tau(s)=\exp(s\gamma)$. Since for all $r\neq k$ it is true $T_{2r}(P_{i-1})\neq 0$, we obtain $\gamma(T_{11})(P_{i-1})\neq 0$. Therefore, there is $s\in\KK$ such that $T_{11}\left(\tau(s)(P_{i-1})\right)\neq 0$. Let us denote $P_{i-1}'=\tau(s)(P_{i-1})$. Since $k\geq i$, we have $T_{2j}(P_{i-1}')=T_{2j}(P_{i-1})=T_{2j}(Q)$ for all $j<i$.
Now we have $\delta_{i+}(T_{2i})(P_{i-1}')\neq 0$. Hence, we can take $P_i=\varphi_{i+}(t)(P'_{i-1}))$ for suitable $t$.

We have obtained a point $R=\psi(P)$ for some $\psi\in\mathrm{SAut}(X)$ such that for all $i$ we have $T_{2i}(R)=T_{2i}(Q)$, for all $j\geq 2$ we have $T_{0j}(R)=T_{0j}(Q)=a_j$ and $T_{1j}(R)=T_{1j}(Q)=b_j$.

Since for all $i$ and~$j$ it is true $T_{ij}(Q)\neq 0$, we have $\delta_{1+}(\alpha)(R)\neq 0$. Hence, there is $\hat{s}$ such that $\alpha(\varphi_{1+}(\hat{s})(R))=\alpha(Q)$. Let us denote $M=\varphi_{1+}(\hat{s})(R)$. Volumes of the following functions at $M$ and $Q$  coincide:
$\alpha$, $T_{0j}$, $T_{1j}$, and $T_{2j}$ for all $j\geq 2$.

We have $\alpha\beta=T_2^{l_2}$. Therefore, $\alpha(Q)\neq 0$. Since volumes of all $T_{0j}, T_{1j}$ for $j\geq 2$, and $\alpha$  at $M$ are not equal to zero, we obtain $\delta_{1-}(T_{21})(M)\neq0$. Hence, we can find $t\in\KK$ such that $T_{21}(\varphi_{1-}(t)(M))=T_{21}(Q)$. From the other hand $\varphi_{1-}(t)$ does not change $T_{ij}$, $j\geq 2$ and $\alpha$. So, volumes of the following functions coincide at $N=\varphi_{1-}(t)(M)$ and $Q$:
$T_{0j}$ and $T_{1j}$ for $j\geq 2$, $T_{2i}$ for all $i$, and $\alpha.$ 
Since $\alpha\beta=T_2^{l_2}$ and $\alpha(Q)\neq 0$, we obtain $\beta (N)=\beta(Q)$. Therefore, all variables coincide at $N$ and $Q$. That is $N=Q$.
\end{proof}

\begin{proof}[Proof of Theorem~\ref{fltr}] In all the items we prove that $X^{reg}$ is one $\mathrm{SAut}(X)$-orbit, where $X$ is the considered trinomial hypersurface. 

{\bf Type $H_1$.} Suppose $P,Q \in X$ such that for all $1\leq i\leq n_2$ we have $T_{2i}(P)\neq 0$ and $T_{2i}(Q)\neq 0$. Then $P$ and $Q$ are in one $\mathrm{SAut}(X)$-orbit. 
Indeed, by Lemma~\ref{nenul} there is $\psi_1\in\mathrm{SAut}(X)$ such that 
$$T_{21}(\psi_1( P))=T_{21}(Q),\qquad \forall j\neq 1\colon T_{2j}(\psi_1( P))=T_{2j}(P).$$
Then again by Lemma~\ref{nenul} there is $\psi_2\in\mathrm{SAut}(X)$ such that 
$$T_{22}((\psi_2\circ\psi_1)(P))=T_{22}(Q),\qquad \forall j\neq 2\colon T_{2j}((\psi_2\circ\psi_1)(P))=T_{2j}(\psi_1( P)).$$ 
And so on. Terminally we obtain, that $(\psi_{n_2}\circ\ldots\circ\psi_1)(P)=Q$.

Now let us take $R\in X^{reg}$. Our goal is to find $\psi\in\mathrm{SAut}(X)$ such that for all $1\leq i\leq n_2$ $T_{2i}(\psi(R))\neq 0$. If there are $i\in\{0,1\}$ and $j\in\{1,2\ldots,n_i\}$ such that 
$$\frac{T_i^{l_i}}{T_{ij}}(R)\neq 0,$$ 
then $\gamma_{ij}(T_{21})(R)\neq 0$. Therefore, we can find $s\in\KK$ such that $T_{21}(\exp(s\gamma_{ij})(R))\neq 0$. We can choose  such $s$ that 
$$\frac{T_i^{l_i}}{T_{ij}}(\exp(s\gamma_{ij})(R))\neq 0.$$ 
But $T_{2k}(\exp(s\gamma_{ij})(R))=T_{2k}(R)$, $k>1$. Analogically we can apply automorphisms from $\mathrm{SAut}(X)$ to get a point $S$ with $T_{2k}(S)\neq 0$ for all~$k$.  

 If there are no $i\in\{0,1\}$ and $j\in\{1,2\ldots,n_i\}$ such that 
 $$\frac{T_i^{l_i}}{T_{ij}}(R)\neq 0.$$ 
 Then for all $i\in\{0,1\}$ and $j\in\{1,2\ldots,n_i\}$ we have
 $$\frac{\partial f}{\partial T_{ij}}=0,$$
 where $f=T_0^{l_0}+T_1^{l_1}+T_{21}T_{22}\ldots T_{2n_2}$. Since $R\in X^{reg}$, there exists $1\leq a\leq n_2$ such that for all $b\neq a$ we have $T_{2b}(R)\neq 0$. 
 Applying Lemma~\ref{nenul} we obtain $\psi\in\mathrm{SAut}(X)$ such that $T_{2i}(\psi(R))\neq 0$ for all $i$.

{\bf Type $H_2$.} Analogically to $H_1$, using Lemma~\ref{dve} we can conclude that if $P,Q\in X^{reg}$ and 
$$T_{0i}(P)\neq 0, T_{0i}(Q)\neq 0, T_{1j}(P)\neq 0, T_{1j}(Q)\neq 0$$ 
for all $1\leq i\leq n_0$, $1\leq j\leq n_1$, then $P$ and $Q$ are in one $\mathrm{SAut}(X)$-orbit. 
Indeed, we can find $\psi_1\in \mathrm{SAut}(X)$ such that 
\begin{equation*}T_{01}(\psi_1(P))=T_{01}(Q), T_{11}(\psi_1(P))=T_{11}(Q),
\end{equation*}
\begin{equation*} 
T_{0i}(\psi_1(P))=T_{0i}(P), T_{1j}(\psi_1(P))=T_{1j}(P) \text{ for all } i,j>1.
\end{equation*}

In this way we can find $\psi\in\mathrm{SAut}(X)$ such that 
\begin{equation*}
\forall i,j\colon T_{0i}(\psi)(P)=T_{0i}(Q), T_{1j}(\psi)(P)=T_{1j}(Q),
\end{equation*}

Therefore, by Lemma~\ref{dve}, $P$ and $Q$ are in one $\mathrm{SAut}(X)$-orbit.

Let $R\in X^{reg}$. Suppose $T_{ij}(R)=0$, for some $i\in\{0,1\}$. We can assume that $T_{01}(R)=0$. Suppose all $l_{2k}>1$. Since $R\in X^{reg}$ we obtain that for all $j$ and $k$ we have $T_{1j}(R)\neq 0$ and $T_{2k}(R)\neq 0$. Therefore, there is $s$ such that $T_{01}(\varphi_{1+}(s)(R))\neq 0$. We can take such $s$ that $T_{1j}(\varphi_{1+}(s)(R))\neq 0$ for all $j$. So, there exists $\psi\in\mathrm{SAut}(X)$ such that for all $i$ and $j$ we have $T_{0i}(\psi(R))\neq 0, T_{1j}(\psi(R))\neq 0$.

If there is $1\leq k\leq n_2$ such that $l_{2k}=1$, then $X$ is of Type $H_4$. We consider this case later.

{\bf Type $H_3$.} Let $P,Q\in X^{reg}$. We can fix $Q$ such that $T_{12}\ldots T_{1n_1}T_{22}\ldots T_{2n_2}(Q)\neq 0$. If $T_{12}\ldots T_{1n_1}(P)\neq 0$, then using Lemma~\ref{nenul} we can obtain $P'=\psi(P)$, $\psi\in\mathrm{SAut}(X)$, $T_{2j}(P')=T_{2j}(Q)$ for all $1\leq j\leq n_2$. By Lemma~\ref{nenul} $P'$ and $Q$ are in one $\mathrm{SAut}(X)$-orbit. Therefore $P$ and $Q$ are in one $\mathrm{SAut}(X)$-orbit. 

If $T_{12}\ldots T_{1n_1}(P)=0$ and $T_{22}\ldots T_{2n_2}(P)\neq 0$ we can swap $T_1^{l_1}$ and $T_2^{l_2}$.

If $T_{12}\ldots T_{1n_1}(P)= 0$ and $T_{22}\ldots T_{2n_2}(P)=0$, then $T_0^{l_0}(P)=0$. Since $P$ is regular, there exist
$a\in\{0,1,2\}$ and $1\leq j\leq n_a$ such that $l_{aj}=1$ and 
$$T_{a1}\ldots T_{aj-1}T_{aj+1}\ldots T_{an_a}(P)\neq 0.$$ 
Then we can swap monomials $a$ and $1$ and swap variables in such a way that $T_{12}\ldots T_{1n_1}(P)\neq0$. So, we obtain that $X^{\mathrm{reg}}$ is one orbit.

{\bf Type $H_4$.} Let $P,Q\in X^{reg}$, $T_{ij}(Q)\neq 0$ for all $i,j$. If $T_{02}\ldots T_{0n_0}(P)\neq 0$, using Lemma~\ref{nenul} we can find $\psi\in\mathrm{SAut}(X)$ such that $T_{ij}(\psi(P))=T_{ij}(Q)$ for $i=1,2$. Then by Lemma~\ref{dve} $\psi(P)$ and $Q$ are in one $\mathrm{SAut}(X)$-orbit.

If among $T_{01},\ldots,T_{0n_0}$ there is only one variable $T_{0j}$ such that $T_{0j}(P)=0$ and $l_{0j}=1$, then we can swap $T_{01}$ and $T_{0j}$. 

If among $T_{01},\ldots,T_{0n_0}$ there are at least two variables vanishing at $P$ or there is only one variable $T_{0j}$ vanishing on $P$ but $l_{0j}>1$, then since $P$ is regular, we have $T_{11}\ldots T_{1n_1}(P)\neq 0$ and $T_{21}\ldots T_{2n_2}(P)\neq 0$. Then applying Lemma~\ref{dve} 
we can obtain $\tau\in\mathrm{SAut}(X)$ such that $T_{02}\ldots T_{0n_0}(\tau(P))\neq0$.

{\bf Type $H_5$.} Let $P, Q\in X^{reg}$, $T_{ij}(Q)\neq 0$ for all $i,j$. Since $P$ is regular, at list for two $i$ we have $T_{i1}\ldots T_{in_i}(P)\neq 0$. 
We can assume that $T_{01}\ldots T_{0n_0}(P)\neq 0$ and $T_{11}\ldots T_{1n_1}(P)\neq 0$. 
By Lemma~\ref{dve} there exists $\psi_1\in\mathrm{SAut}(X)$ such that for all $2\leq i\leq n_0$, $2\leq j\leq n_1$ and $1\leq k\leq n_2$ we have
$$T_{0i}(\psi_1(P))=T_{0i}(P),\qquad T_{1j}(\psi_1(P))=T_{1j}(P),\qquad T_{2k}(\psi_1(P))=T_{2k}(Q).$$
Applying once more Lemma~\ref{dve} we obtain $\psi_2\in\mathrm{SAut}(X)$ such that for all $1\leq i\leq n_0$, $2\leq j\leq n_1$ and $2\leq k\leq n_2$ we have
\begin{equation*}
T_{0i}(\psi_2\circ\psi_1(P))=T_{0i}(Q),\qquad T_{1j}((\psi_2\circ\psi_1(P))=T_{1j}(P),
\end{equation*} 
\begin{equation*}
T_{2k}((\psi_2\circ\psi_1(P))=T_{2k}(\psi_1(P))=T_{2k}(Q).
\end{equation*} 
Again by Lemma~\ref{dve} $\psi_2\circ\psi_1(P)$ and $Q$ are in one $\mathrm{SAut}(X)$-orbit.

The proof of Theorem~\ref{fltr} is completed.

\end{proof}

\begin{re}
Theorem~\ref{fltr} does not give a criterium for a trinomial hypersurface to be flexible, but only a sufficient condition.
\end{re}
\begin{re}
Flexibility of varieties of Type $H_1$ when $n_2=2$ follows from results of \cite{AKZ}.
\end{re}

\begin{ex}
In \cite[Theorem~9.2]{FMJ} it is proved that $\mathrm{ML}\left(\VV(t^2x^2+y^2+z^n)\right)=\KK$. This variety is of Type $H_2$. Hence, it is flexible by Theorem~\ref{fltr}.
\end{ex}

\section{Intermediate trinomial hypersurfaces}

Theorem~\ref{rt} gives a criterium for a trinomial hypersurface to be rigid. Theorem~\ref{fltr} gives a sufficient condition for a trinomial hypersurface to be flexible.
Let us consider varieties, which are not rigid and are not isomorphic to any variety from Theorem~\ref{fltr}. There are two types of them

$
\text{\bf Type A.}\qquad \VV\left(T_0^{l_0}+T_1^{l_1}+T_{21}\ldots T_{2k}T_{2k+1}^{l_{2k+1}}\ldots T_{2n_2}^{l_{2n_2}}\right),
$
where $k\geq 1$, $n_2-k\geq 1$, all $l_{ij}\geq 2$, and 
$$\{l_{01},\ldots,l_{0n_0}\}\neq \{2, 2a_1,\ldots, 2a_{n_0-1}\}.$$

$
\text{\bf Type B.}\qquad
\VV\left(T_{01}^2\ldots T_{0k_0}^2T_{0k_0+1}^{2m_{0k_0+1}}\ldots T_{0n_0}^{2m_{0n_0}}+T_{11}^2\ldots T_{1k_1}^2 T_{1k_1+1}^{2m_{1k_1+1}}\ldots T_{1n_1}^{2m_{1n_1}}+T_2^{l_2}\right),
$
where $k_0,k_1\geq 1$, $n_0-k_0+n_1-k_1\geq 1$, all $m_{ij}\geq 2$, all $l_{2i}\geq 2$, and 
$$\{l_{21},\ldots,l_{2n_2}\}\neq \{2, 2a_1,\ldots, 2a_{n_2-1}\}. $$

\begin{de}
A variety, which is not rigid and is not flexible, we call {\it intermediate}.
\end{de}

Let us formulate a question, which remains open.
\begin{quest}\label{qu}
Does Theorem~\ref{fltr} give a criterium for a trinomial hypersurface to be flexible? That is, is it true that trinomial hypersurfaces of Types A and B are intermediate?
\end{quest}

The next two propositions give a partial answer to Question~\ref{qu} when the trinomial hypersurface is of Type A and $k=1$ and when it is of Type B and $k_0=1$. In these cases the answer is "yes". This implies that there exists an intermediate trinomial hypersurface. 

\begin{prop}\label{m}
Let 
$$X=\VV\left(T_0^{l_0}+T_1^{l_1}+T_{21}T_{22}^{l_{22}}\ldots T_{2n_2}^{l_{2n_2}}\right), $$ 
$n_2\geq 2$, all $l_{ij}\geq 2$ and $\{l_{01},\ldots,l_{0n_0}\}\neq \{2, 2a_1,\ldots, 2a_{n_0-1}\}.$ Then 

1) $\mathrm{ML}(X)=\KK[T_{22}, \ldots, T_{2n_2}]$,

2) generic $\mathrm{SAut}(X)$-orbits on $X$ can be separated by functions from $\mathrm{ML}(X)$. 

\end{prop} 
\begin{proof}{\bf 1)}
Note that $X$ is an $n_2$-suspension over $\KK^{n_0+n_1}$. 

Consider
$$
Z=\VV\left(T_0^{l_0}+T_1^{l_1}+S^3T_{22}^{l_{22}}\ldots T_{2n_2}^{l_{2n_2}}\right).
$$
Theorem~\ref{rt} implies that $Z$ is rigid. Hence, Lemma~\ref{tool} implies that for all $2\leq j\leq n_2$ we have $T_{2j}\in\mathrm{ML}(X)$. 
Therefore, $\KK[T_{22}, \ldots, T_{2n_2}]\subset\mathrm{ML}(X)$.

In the proof of Lemma~\ref{nenul} we define an LND $\gamma_{ij}\colon\KK[X]\rightarrow\KK[X]$. We have 
$$\Ker\gamma_{ij}=\KK\left[T_{ab}\mid (a,b)\neq  (2,1), (a,b)\neq (i,j)\right].$$ 
Indeed, for any nonzero LND $\partial$ of $\KK[X]$ we have a formula for transcendence degree of its kernel
$$\mathrm{tr.deg}\,\Ker\partial=\mathrm{tr.deg}\,\KK[X]-1=n_0+n_1+n_2-2,$$ 
see for example \cite[Principle~11(e)]{Fr}. It is obvious that 
$$
\KK\left[T_{ab}\mid (a,b)\neq  (2,1), (a,b)\neq (i,j)\right]\subset\Ker\gamma_{ij}.
$$
If there is a polynomial depending on $T_{21}$ or $T_{ij}$ in $\Ker\gamma_{ij}$, then 
$$\mathrm{tr.deg}\Ker\gamma_{ij}>n_0+n_1+n_2-2,$$ this gives a contradiction.

We obtain 
$$\mathrm{ML}(X)\subset \bigcap_{i=1,2;\,j=1,\ldots,n_i}\Ker\gamma_{ij}=\KK[T_{22}, \ldots, T_{2n_2}].$$

{\bf 2)} follows from Lemma~\ref{nenul}.
\end{proof}

\begin{re}
$\mathrm{SAut}(X)$-orbits on an arbitrary affine variety can not always be separated by $\mathrm{ML}(X)$. They always can be separated by rational $\mathrm{SAut}(X)$-invariants $\mathrm{FML}(X)$, see \cite[Theorem~1.13]{AFKKZ}. For example, there are nonflexible affine threefolds with trivial Makar-Limanov invariant, see \cite[$\S$ 4.2]{Li}.  
\end{re}

\begin{prop}\label{mm}
Let 
$$
X=\VV\left(T_{01}^2T_{02}^{2m_{02}}\ldots T_{0n_0}^{2m_{0n_0}}+T_{11}^2\ldots T_{1k_1}^2 T_{1k_1+1}^{2m_{1k_1+1}}\ldots T_{1n_1}^{2m_{1n_1}}+T_2^{l_2}\right),
$$
where $n_0\geq 2$, $1\leq k_1\leq n_1$, for all $i$, $j$, $k$ we have $l_{2i}\geq 2$ and $m_{jk}\geq 2$, and
$$\{l_{21},\ldots,l_{2n_2}\}\neq \{2, 2a_1,\ldots, 2a_{n_2-1}\}. $$
Then 

1) $\mathrm{ML}(X)\supset\KK[T_{02},\ldots, T_{0n_0}]$,

2) If $k_1=1$, then 

\ \ \ a) $\mathrm{ML}(X)=\KK[T_{02},\ldots, T_{0n_0}, T_{12}, \ldots, T_{1n_1}],$

\ \ \ b) generic $\mathrm{SAut}(X)$-orbits can be separated by $\mathrm{ML}(X)$. 

3) If $k_1=n_1$, then 

\ \ \ a) $\mathrm{ML}(X)=\KK[T_{02},\ldots, T_{0n_0}],$

\ \ \ b) generic $\mathrm{SAut}(X)$-orbits can be separated by $\mathrm{ML}(X)$. 

\end{prop}
\begin{proof}
{\bf 1)} Note that $X$ is an $n_0$-suspension over $\KK^{n_1+n_2}$. Let us consider 
$$
Z=\VV\left(S^{2u}T_{02}^{2m_{02}}\ldots T_{0n_0}^{2m_{0n_0}}+T_{11}^2\ldots T_{1k_1}^2 T_{1k_1+1}^{2m_{1k_1+1}}\ldots T_{1n_1}^{2m_{1n_1}}+T_2^{l_2}\right),
$$ 
where $u\geq 2$.
Theorem~\ref{rt} implies that $Z$ is rigid. It follows from Lemma~\ref{tool} that for all $j\geq 2$ we have $T_{0j}\in \mathrm{ML}(X)$.

{\bf 2)} By 1) we have $\mathrm{ML}(X)\supset\KK[T_{02},\ldots, T_{0n_0}, T_{12}, \ldots, T_{1n_1}]$. Lemma~\ref{dve} implies that a generic variety 
$$\VV\left(T_{02}-a_2,\ldots, T_{0n_0}-a_{n_0}, T_{12}-b_1, \ldots, T_{1n_1}-b_{n_1}\right)\cap X^{\mathrm{reg}}$$ 
is one $\mathrm{SAut}(X)$-orbit. This implies 2(a) and 2(b).  

{\bf 3)} By 1) we have $\mathrm{ML}(X)\supset\KK[T_{02},\ldots, T_{0n_0}]$. Let 
$$
U=\{x\in X\mid \forall i,j\colon T_{ij}(x)\neq 0\}.
$$
It is clear that $U\subset X$ is open. Then $V=\mathrm{SAut}(X)(U)\subset X$ is open.

Let us prove that  every nonempty set
$$W(a_1,\ldots,a_{n_0})=\VV\left(T_{02}-a_2,\ldots, T_{0n_0}-a_{n_0}\right)\cap V$$ 
is one $\mathrm{SAut}(X)$-orbit. 

Since $\mathrm{ML}(X)\supset\KK[T_{02},\ldots, T_{0n_0}]$, if $W(a_1,\ldots,a_{n_0})\neq\varnothing$, then $W(a_1,\ldots,a_{n_0})\cap U\neq\varnothing$. 
Let $P, Q\in W(a_1,\ldots,a_{n_0})$. Then there are $\psi_P$ and $\psi_Q$ in $\mathrm{SAut}(X)$ such that $P'=\psi_P(P)$ and $Q'=\psi_Q(Q)$ are in 
$W(a_1,\ldots,a_{n_0})\cap U$. 

By Lemma~\ref{dve} there is $\psi_1\in\mathrm{SAut}(X)$ such that 
$$
T_{11}(\psi_1(P'))=T_{11}(Q'),\qquad \forall\, i\geq 2\colon T_{1i}(\psi_1(P'))=T_{1i}(P').
$$

Applying again Lemma~\ref{dve} we obtain $\psi_2\in\mathrm{SAut}(X)$ such that 
$$
T_{11}(\psi_2\circ\psi_1(P'))=T_{11}(Q'),\qquad T_{12}(\psi_2\circ\psi_1(P'))=T_{12}(Q'),
$$
$$ 
\forall\, i\geq 3\colon T_{1i}(\psi_2\circ\psi_1(P'))=T_{1i}(P').
$$
We proceed in such a way until we obtain $\psi \in\mathrm{SAut}(X)$ such that for all $i=1,\ldots,n_1$ we have $T_{1i}(\psi(P'))=T_{1i}(Q')$. Since for all $j\geq 2$ we have $T_{0j}(\psi(P'))=T_{0j}(P')=T_{0j}(Q')$, Lemma~\ref{dve} implies existing of $\tau\in\mathrm{SAut}(X)$ such that $(\tau\circ\psi)(P')=Q'$. This implies that $W(a_1,\ldots,a_{n_0})$ is one $\mathrm{SAut}(X)$-orbit. This proves 3(a) and 3(b).

\end{proof}

\end{document}